\documentclass[reqno]{amsart}
\usepackage{amssymb,amsmath,amsfonts,amsthm,latexsym}

\newtheorem{thm}{Theorem}[section]
\newtheorem{cor}[thm]{Corollary}

\newtheorem{prop}[thm]{Proposition}
\newtheorem{defn}[thm]{Definition}
\newtheorem{rem}[thm]{Remark}
\newtheorem{Rem}[thm]{Remark}
\numberwithin{equation}{section}

\DeclareMathOperator{\Ann}{Ann}

\begin{document}
\title[classification of solvable Leibniz algebras with null-filiform nilradical] {classification of solvable Leibniz algebras with null-filiform nilradical}
\author{J. M. Casas, M. Ladra, B. A. Omirov and I.A. Karimjanov}
\address{[J.M. Casas] Department of Applied Mathematics, University of Vigo, E.U.I.T. Forestal, Pontevedra 36005, Spain} \email{jmcasas@uvigo.es}
\address{[M. Ladra] Department of Algebra, University of Santiago de Compostela, Santiago de Compostela 15782, Spain} \email{manuel.ladra@usc.es}
\address{[B.A. Omirov --- I.A. Karimdjanov] Institute of Mathematics and Information Technologies  of Academy of Uzbekistan, 29, Do'rmon yo'li street., 100125, Tashkent (Uzbekistan)} \email{omirovb@mail.ru --- iqboli@gmail.com}



\begin{abstract}
In this paper we classify solvable Leibniz algebras whose nilradical is a null-filiform algebra. We extend the obtained classification to the case when the solvable Leibniz algebra is decomposed as a direct sum of its nilradical, which is a direct sum of null-filiform ideals, and a one-dimensional complementary subspace. Moreover, in this case  we establish that these ideals are ideals of the algebra, as well.
\end{abstract}

\maketitle
\section{Introduction}

The notion of Leibniz algebra has been firstly introduced by Loday
in \cite{Lod} as a non-antisymmetric generalization of Lie
algebras. During the last $20$ years the theory of Leibniz algebras has
been actively studied and many results of the theory of Lie
algebras have been extended to Leibniz algebras. For instance, the classical results on Cartan subalgebras, regular elements and others from the theory of Lie algebras are also true for Leibniz algebras \cite{AlAyBa,Omi}.

From classical theory of finite-dimensional Lie algebras it is known that an arbitrary Lie algebra is decomposed into a semidirect sum of the solvable radical and a semisimple subalgebra (Levi's theorem). In addition, the semisimple part is a direct sum of simple ideals, which is completely classified \cite{Jac}. Thanks to Mal'cev's results \cite{Mal} the study of solvable Lie algebras is reduced to the study of nilpotent ones. Thus, the description of finite-dimensional Lie algebras is reduced to the description of nilpotent algebras.

In the case of Leibniz algebras, the analogue of Levi's theorem  was proven in \cite{Bar}. Namely, a Leibniz algebra is decomposed into a semidirect sum of the  solvable radical and a semisimple Lie algebra.
Thus, the semisimple part can be described from simple Lie ideals and the main problem is to study the solvable radical, i.e. the description of solvable Leibniz algebras. The inherent properties of non-Lie Leibniz algebras imply that the subspace spanned by squares of elements of the algebra is a non-trivial  ideal (moreover, this ideal is abelian). In fact, this ideal is the minimal one such that the quotient algebra is a Lie algebra. Thus, we also reduce our study of Leibniz algebras to the solvable ones.

The investigation of solvable Lie algebras with some special types of nilradical comes from different problems in Physics and was the subject of various papers \cite{AnCaGa1,AnCaGa2,BoPaPo,Cam,NdWi,SnKa,SnWi,TrWi,WaLiDe} and many other references given there.
Also it is natural to add  restrictions to the index of nilpotency and graduation on the nilradical. For example,  the cases where the nilradical of a solvable Lie algebra is filiform, quasi-filiform and abelian  were considered. We recall that the maximal index of nilpotency of an $n$-dimensional Lie algebra is $n$ (such algebras were called  filiform in  \cite{Ver}).
However,  the maximal index of nilpotency of an $n$-dimensional Leibniz algebra is equal to $n+1$ (such algebras  were called  null-filiform in \cite{AyOm2}).

Our goal in the present paper is to classify solvable Leibniz algebras with null-filiform nilradical. Moreover, this classification is extended to the case when the nilradical is a direct sum of null-filiform ideals and the complementary vector space of the nilradical is one-dimensional.



In order to achieve our goal, the paper is organized as follows. In Section~\ref{S:2} we recall some needed notions and properties of Leibniz algebras. We start  Section~\ref{S:3} establishing that the dimension of a solvable Leibniz algebra whose nilradical is an $n$-dimensional null-filiform Leibniz algebra is exactly $n+1$; after that,  we present  our main results: the classification of solvable Leibniz algebras that  can be decomposed as a direct sum of their nilradical and a complementary vector space, where the nilradical is a direct sum of null-filiform Leibniz algebras. Firstly we study the case when the solvable Leibniz algebra is a direct sum of its nilradical and a one-dimensional complementary vector space, where the nilradical is null-filiform; after that we consider the case where the nilradical decomposes in a direct sum of two null-filiform ideals. Finally we consider the general situation where the nilradical decomposes as a direct sum of any null-filiform ideals.

Throughout the paper we consider finite-dimensional vector spaces and algebras over a field of characteristic zero. Moreover, in the table of  multiplication  of an algebra omitted products are assumed to be zero and if it is not noticed we shall consider non-nilpotent solvable algebras.

\section{Preliminaries} \label{S:2}
In this section we give necessary definitions and preliminary results.

\begin{defn} An algebra $(L,[-, -])$ over a field $F$ is said to be a Leibniz algebra if for any $x,y,z\in L$ the so-called Leibniz identity
\[ \bigl[x,[y,z]\bigr]=\bigl[[x,y],z\bigr] - \bigl[[x,z],y\bigr] \] holds.
\end{defn}

A subalgebra $H$ of a Leibniz algebra $L$ is said to
be a \emph{two-sided ideal} if $[L,H] \subseteq L$ and $[H,L] \subseteq L$.
 Let $H$ and $K$ be two-sided ideals of a Leibniz algebra $L$. The
\emph{commutator ideal} of $H$ and $K$, denoted by
$[H,K]$ is the two-sided ideal of
$L$ spanned by the brackets $[h, k], [k, h], h \in H, K \in K$. Obviously
$[H,K] \subset H \cap K$.

From the Leibniz identity we conclude that the elements of the form $[x,x]$ and $[x,y]+[y,x]$, for any $x, y$, lie in the  \emph{right annihilator} $\Ann_r(L) = \{ x \in L: [y,x]=0 \ {\rm for\ all\ } y \in L \}$ of the Leibniz algebra. Moreover, we also get that $\Ann_r(L)$ is a two-sided ideal of the Leibniz algebra.

For a given Leibniz algebra $L$, we define the \emph{lower central} and  \emph{derived series} to the sequences of two-sided ideals defined recursively as follows:
\[L^1=L, \ L^{k+1}=[L^k,L],   \ \ k \geq 1; \qquad\qquad
L^{[1]}=L, \ L^{[s+1]}=[L^{[s]},L^{[s]}],  \ \ s \geq 1.
\]
\begin{defn} A Leibniz algebra $L$ is said to be
nilpotent (respectively, solvable), if there exists $n\in\mathbb N$ ($m\in\mathbb N$) such that $L^{n}=0$ (respectively, $L^{[m]}=0$).
The minimal number $n$ (respectively, $m$) with such property is said to be the index of
nilpotency (respectively, of solvability) of the algebra $L$.
\end{defn}

\begin{Rem}
Obviously, the index of nilpotency of an $n$-dimensional nilpotent Leibniz algebra is not greater than $n+1$.
\end{Rem}

\begin{defn} An $n$-dimensional Leibniz algebra is said to be null-filiform if $\dim L^i=n+1-i, \ 1\leq i \leq n+1$.
\end{defn}

\begin{Rem}
Obviously, a null-filiform Leibniz algebra has maximal index of nilpotency.
\end{Rem}

\begin{thm}[\cite{AyOm2}] \label{t21} An arbitrary $n$-dimensional null-filiform Leibniz algebra is isomorphic to the algebra:
\[NF_n: \quad [e_i, e_1]=e_{i+1}, \ 1 \leq i \leq n-1,\]
where $\{e_1, e_2, \dots, e_n\}$ is a basis of the algebra $NF_n$.
\end{thm}

From this theorem it is easy to see that a nilpotent Leibniz algebra is null-filiform if and only if it is a one-generated algebra. Note that this notion has no sense in Lie algebras case, because they are at least two-generated.

\begin{defn} The  maximal nilpotent ideal of a Leibniz algebra is said to be the nilradical of the algebra.
\end{defn}

\begin{defn}
For a Leibniz algebra $L$, a linear map $d \colon  L \to L$ is said to be a derivation if \[d[x,y]=[d(x),y]+[x,d(y)]\] for all $x, y \in L$.
\end{defn}

For a fixed $x \in L$, the map $R_x \colon  L \to L, R(x)(y)=[y,x]$ is a derivation. We call this kind of derivations as  \emph{inner derivations}. Derivations that are not inner are said to be  \emph{outer derivations}.

\begin{defn}[\cite{Mub}]
Let $d_1, d_2, \dots, d_n$ be derivations of a Leibniz algebra $L$. The  derivations $d_1, d_2, \dots, d_n$ are said to be nil-independent if  \[\alpha_1 d_1 + \alpha_2 d_2 + \dots + \alpha_n d_n\]
is not nilpotent for any scalars $\alpha_1, \alpha_2, \dots, \alpha_n \in F$.

In other words, if for any $\alpha_1, \alpha_2, \dots, \alpha_n \in F$ there exists a natural number $k$ such that  $\left( \alpha_1 d_1 + \alpha_2 d_2 + \dots + \alpha_n d_2 \right)^k = 0$, then  $\alpha_1 = \alpha_2 = \dots = \alpha_n = 0$.
\end{defn}

\section{Main results}\label{S:3}

Let $R$ be a solvable Leibniz algebra. Then it can be decomposed in the form $R=N \oplus Q$, where $N$ is the  nilradical and $Q$ is the complementary vector space. Since the square of a solvable algebra is a nilpotent ideal and the finite sum of nilpotent ideals is a nilpotent ideal too \cite{AyOm1}, then we get the nilpotency of the ideal $R^2$, i.e. $R^2\subseteq N$ and consequently, $Q^2\subseteq N$.

\begin{thm} \label{t31} Let $R$ be a solvable Leibniz algebra and $N$ its nilradical. Then the dimension of the  complementary vector  space to $N$ is not greater than the  maximal number of nil-independent derivations of $N$.
\end{thm}
\begin{proof}
Let us consider the restrictions of the right multiplication operators on an element $x \in Q$ (denoted by $R_{{x  \mid}_{N}}$).
 We assert that every $R_{{x \mid}_{N}}$  is a non-nilpotent derivation of $N$.  Indeed, if there exists some $x \in Q$ such that the operator $R_{{x \mid}_{N}}$ is nilpotent, then the subspace $\langle  x + N  \rangle$ is a nilpotent ideal of the algebra $R$.  $\langle  x + N  \rangle$ is an ideal because of $R^2 \subseteq N$; its nilpotency is argued as follows:  if $a \in N$, then $R_a$ is a nilpotent operator thanks to Engel's theorem for Leibniz algebras \cite{AyOm1}; since $R_{{x  \mid}_{N}}$ is a nilpotent operator by assumption, then $R_a + R_{{x  \mid}_{N}}$ is a nilpotent operator because of the formula
  \[(R_a + R_{{x  \mid}_{N}})^n = \sum_{i=0}^n \dbinom{n}{i}   R_a^i R_{{x  \mid}_{N}}^{n-i}\]
If $s$ and $t$ are the indices of nilpotency of $R_a$ and $R_{{x  \mid}_{N}}$, respectively, then it is enough to take $n=s+t$. Consequently,  $\langle  x + N  \rangle$  is a nilpotent ideal, which is  in contradiction with the maximality condition of $N$.

Thus, we obtain that for any $x$ from $Q$, the operator $R_{{x  \mid}_{N}}$ is a non-nilpotent outer derivation of $N$.

Let us assume that  $\{x_1, \dots, x_m\}$ is a basis of $Q$. Then the operators $R_{{x_1  \mid}_{N}}, \dots,R_{{x_m \mid}_{N}}$ are nil-independent, since if for some scalars $\{\alpha_1, \dots, \alpha_m\}\in F\setminus \{0\}$ we have that
$\Bigl( \displaystyle \sum_{i=1}^m \alpha_iR_{{x_i  \mid}_{N}} \Bigr)^k =0$, then  $\Bigl( R_{\Bigl( \displaystyle \sum_{i=1}^m \alpha_i x_i \Bigr)\!\!\bigm|_{N}} \Bigr)^k =0$, i.e. for the element  $y=\displaystyle \sum_{i=1}^m \alpha_i x_i$, the operator $R_{{y \mid}_{N}}$ is nilpotent, then $y=0$, and so  $\alpha_i =0$ for $i=1, \dots, m$.

Therefore, we have that dimension of $Q$ is bounded by the maximal number of nil-independent derivations of the nilradical $N$. Moreover, similarly to the case of Lie algebras, for solvable Leibniz algebra $R$ we also have inequality $\dim N \geq \dfrac{\dim R}{2}$.
\end{proof}

From Theorem~\ref{t31}, they can be derived the following consequences about the derivations of null-filiform Leibniz algebras.

\begin{prop} Any derivation of the algebra $NF_n$ has the following matrix form:
\[
\begin{pmatrix}
a_1& a_2&a_3&\ldots&a_n\\
0& 2a_1&a_2&\ldots& a_{n-1}\\
0& 0&3a_1&\ldots& a_{n-2}\\
\vdots&\vdots&\vdots&\vdots&\vdots\\
0&0&0&\ldots&na_1
\end{pmatrix}.
\]
\end{prop}
\begin{proof}
The proof is carried out by checking the derivation property on algebra $NF_n$.
\end{proof}

\begin{cor} \label{cor}
The  maximal number of nil-independent derivations of the  $n$-dimensional null-filiform Leibniz algebra $NF_n$ is 1.
\end{cor}
\begin{proof}
Let
\[ D_i =
\begin{pmatrix}
a_1^i& a_2^i&a_3^i&\ldots&a_n^i\\
0& 2a_1^i&a_2^i&\ldots& a_{n-1}^i\\
0& 0&3a_1^i&\ldots& a_{n-2}^i\\
\vdots&\vdots&\vdots&\vdots &\vdots\\
0&0&0&\ldots&na_1^i
\end{pmatrix}, \quad i= 1, 2, \dots, p,
\]
be derivations of $NF_n$. $\{D_1, D_2, \dots, D_p \}$ cannot be nil-independent, since there exists at least the linear combination $\left( D_i - \frac{a_1î}{a_1^1} D_1 \right)^n =0$ with non-trivial scalars, except in case $p=1$.
\end{proof}

\begin{cor} \label{cor34} The dimension of a solvable Leibniz algebra with nilradical $NF_n$ is equal to $n+1$.
\end{cor}
\begin{proof}
Let us assume that the solvable Leibniz algebra is decomposed as $R = NF_n \oplus Q$. By Corollary~\ref{cor} and Theorem~\ref{t31} we have $1 \leq \dim Q \leq 1$.
\end{proof}

\begin{thm} \label{thm35} Let $R$ be a solvable Leibniz algebra whose nilradical is $NF_n$. Then there exists a basis $\{e_1, e_2, \dots, e_n, x\}$ of the algebra $R$ such that the multiplication table of $R$ with respect to this basis has the following form:
\[\left\{ \begin{aligned}
{}[e_i,e_1] & =e_{i+1}, && 1\leq i\leq n-1,\\
 [x,e_1]& =e_1, && \\
[e_i,x] & =-ie_i, && 1\leq i\leq n.
\end{aligned}\right.\]
\end{thm}
\begin{proof}
According to Theorem~\ref{t21} and Corollary~\ref{cor34}  there exists  a basis $\{e_1, e_2, \dots, e_n, x\}$ such that all products of basic elements, except for the products $[e_i,x]$ which can be derived from the equalities $[e_{i+1},x]=\bigl[[e_i,e_1],x \bigr]= \bigl[e_i,[e_1,x] \bigr]+\bigl[[e_i,x],e_1 \bigr]$, $1\leq i\leq n-1$, have the following form:
\[\left\{ \begin{aligned}
{} [e_i,e_1]& =e_{i+1}, && 1\leq i\leq n-1, \\
[x,e_1]& =\sum\limits_{i=1}^n\alpha_ie_i, && \\
 [e_1,x]& =\sum\limits_{i=1}^n\beta_ie_i, &&\\
[x,x]& =\sum\limits_{i=1}^n\gamma_ie_i.&&
\end{aligned}\right.\]
where $\{e_1, e_2, \dots, e_n\}$ is a basis of $NF_n$ and $\{x \}$ is a basis of $Q$.

Now we focus our discussion considering the following two possible cases.

\textbf{Case 1.} Let $\alpha_1\neq 0$. Then taking the change of basis:
\[e_i'=\frac{1}{\alpha_1}\sum\limits_{j=i}^n\alpha_{j-i+1}e_j, \quad
1\leq i\leq n, \qquad x'=\frac{1}{\alpha_1}x\,,\]
we can assume that $[x,e_1]=e_1$ and other products by redesignation of parameters can be assumed not changed.

From the products
\[0=\bigl[x,[x,x]\bigr]=\bigl[x,\sum\limits_{i=1}^n\gamma_ie_i\bigr]=
\sum\limits_{i=1}^n\gamma_i[x,e_i]=\gamma_1e_1,\]
we can deduce that $\gamma_1=0$.

On the other hand, from the  Leibniz identity \[\bigl[x,[e_1,x]\bigr]=\bigl[[x,e_1],x\bigr]-\bigl[[x,x],e_1\bigr]\]
we get $\beta_1[x,e_1]=[e_1,x]-\sum\limits_{i=3}^n \gamma_{i-1}e_i$, i.e.
$\beta_1e_1=\sum\limits_{i=1}^n\beta_ie_i-\sum\limits_{i=3}^n
\gamma_{i-1}e_i$.

Comparing coefficients at the basic elements we obtain $\beta_2=0$ and $\gamma_i=\beta_{i+1}$ for $2\leq i\leq n-1$.

From the equality $\bigl[e_1,[e_1,x]\bigr]=-\bigl[e_1,[x,e_1]\bigr]$, we derive that $\beta_1=-1$.

Thus, we have \[[e_1,x]=-e_1+\sum\limits_{i=3}^n\beta_ie_i, \qquad  \qquad  [x,x]=\sum\limits_{i=2}^{n-1}\beta_{i+1}e_i+\gamma_n e_n.\]
No we are going to prove the following identities by induction:
\begin{equation} \label{E:3.1}
[e_i,x]=-ie_i+\sum\limits_{j=i+2}^n\beta_{j-i+1}e_j, \qquad 1\leq i\leq n.
\end{equation}
For $i=1$ equality \eqref{E:3.1} is true.
Let us assume that equalities \eqref{E:3.1} are true for each $i\ (1 \leq i<n)$.

The following chain of equalities
\begin{align*}
[e_{i+1},x] & =\bigl[[e_i,e_1],x\bigr]=\bigl[e_i,[e_1,x]\bigr]+\bigl[[e_i,x],e_1\bigr]  =[e_i,-e_1]+
\bigl[-ie_i+\sum\limits_{j=i+2}^n\beta_{j-i+1}e_j,e_1\bigr]  \\ &{} =-e_{i+1}-ie_{i+1}+\sum\limits_{j=i+2}^n\beta_{j-i+1}[e_j,e_1]=
-(i+1)e_{i+1}+\sum\limits_{j=i+3}^n\beta_{j-i}e_j
\end{align*}
completes the proof of equalities \eqref{E:3.1} for any $i \ (1\leq i\leq n)$.

Thus, the multiplication table of the algebra $R$ has the form:
\begin{equation} \label{E:3.2}
\left\{\begin{aligned}
 {}[e_i,e_1]& =e_{i+1}, && 1\leq i\leq n-1, \\
 [x,e_1]& =e_1, && \\
 [e_i,x]&=-ie_i+\sum\limits_{j=i+2}^n\beta_{j-i+1}e_j,&& 1\leq i\leq n,\\
 [x,x]&=\sum\limits_{i=2}^{n-1}\beta_{i+1}e_i+\gamma_n e_n.&&
\end{aligned}\right.
\end{equation}
Let us take the change of basis:
\[e_i'=e_i+\sum\limits_{j=i+2}^nA_{j-i+1}e_j, \ 1\leq i\leq n,  \qquad x'=\sum\limits_{i=2}^{n-1}A_{i+1}e_i+B_ne_n+x,\]
where parameters $A_i, B_n$ are as follows
\begin{align*}
A_3=\frac{1}{2}\beta_3, \quad A_4=\frac{1}{3}\beta_4, \quad & A_i=\frac{1}{i-1}\Bigl(\sum\limits_{j=3}^{i-2}A_{i-j+1}\beta_j+\beta_i\Bigr), \quad (5\leq
i\leq n), \\ &  B_n=\frac{1}{n}\Bigl(\sum\limits_{j=3}^{n-1}A_{n-j+2}\beta_j+\gamma_n\Bigr).
\end{align*}
Then taking into account the multiplication table \eqref{E:3.2} we calculate the products in new basis
\begin{align*}
[e_i',e_1']& =\bigl[e_i+\sum\limits_{j=i+2}^nA_{j-i+1}e_j,e_1\bigr]=e_{i+1}+
\sum\limits_{j=i+3}^nA_{j-i}e_j=e_{i+1}', \quad 1\leq i\leq n-1, \\
[x',e_1']& =\bigl[\sum\limits_{i=2}^{n-1}A_{i+1}e_i+B_ne_n+x,e_1\bigr]=
\sum\limits_{i=3}^nA_ie_i+[x,e_1]=e_1+\sum\limits_{i=3}^nA_ie_i=e_1',\\
[x',x']& =\bigl[\sum\limits_{i=2}^{n-1}A_{i+1}e_i+B_ne_n+x,x\bigr]=
\sum\limits_{i=2}^{n-1}A_{i+1}[e_i,x]+B_n[e_n,x]+[x,x]\\
&{}=
\sum\limits_{i=2}^{n-1}A_{i+1}\bigl(-ie_i+\sum\limits_{j=i+2}^n\beta_{j-i+1}e_j\bigr)-nB_ne_n+
\sum\limits_{i=2}^{n-1}\beta_{i+1}e_i+\gamma_n e_n \\ &{}=-\sum\limits_{i=2}^{n-1}iA_{i+1}e_i
+
\sum\limits_{i=2}^{n-3}A_{i+1}\sum\limits_{j=i+2}^{n-1}\beta_{j-i+1}e_j
+\sum\limits_{i=2}^{n-1}\beta_{i+1}e_i\\&{}\qquad  \qquad  \qquad \quad +  \sum\limits_{i=2}^{n-2}A_{i+1}\beta_{n-i+1}e_n-B_ne_n+\gamma_n e_n\\
&{}=
\sum\limits_{i=2}^{n-1}(-iA_{i+1}+\beta_{i+1})e_i+
\sum\limits_{i=4}^{n-1}\sum\limits_{j=3}^{i-1}A_{i-j+2}\beta_je_i+
(-nB_n+\gamma_n+\sum\limits_{i=2}^{n-1}A_{i+1}\beta_{n-i+1})e_n \\
&{}=
(-2A_3+\beta_3) + (-3A_4+\beta_4)+\sum\limits_{i=4}^{n-1}\sum\limits_{j=3}^{i-1}(-iA_{i+1}+\beta_{i+1}+
A_{i-j+2}\beta_j)e_i=0,
\end{align*}
\begin{align*}
[e_1',x']&=\bigl[e_1+\sum\limits_{i=3}^nA_ie_i,x\bigr] =[e_1,x]+
\sum\limits_{i=3}^nA_i[e_i,x]\\
{}&=-e_1+\sum\limits_{i=3}^n\beta_ie_i+
\sum\limits_{i=3}^nA_i\Bigl(-ie_i+\sum\limits_{j=i+2}^n\beta_{j-i+1}e_j\Bigr) \\
&{}=
-e_1+\sum\limits_{i=3}^n\beta_ie_i-\sum\limits_{i=3}^niA_ie_i+
\sum\limits_{i=3}^nA_i\sum\limits_{j=i+2}^n\beta_{j-i+1}e_j \\
&{}=
-e_1-\sum\limits_{i=3}^nA_ie_i-\sum\limits_{i=3}^n(i-1)A_ie_i+
\sum\limits_{i=3}^n\beta_ie_i+\sum\limits_{i=3}^n\Bigl(\sum\limits_{j=3}^{i-2}A_{i-j+1}b_j\Bigr)e_i \\
&{}=-e_1-\sum\limits_{i=3}^nA_ie_i+
\sum\limits_{i=3}^n\bigl(-(i-1)A_i+\beta_i\bigr)e_i+
\sum\limits_{i=5}^n\sum\limits_{j=3}^{i-2}A_{i-j+1}\beta_je_i\\
&{}=
-e_1-\sum\limits_{i=3}^nA_ie_i + (-2A_3+\beta_3)e_3+(-3A_4+\beta_4)e_4\\
&+\sum\limits_{i=5}^n\sum\limits_{j=3}^{i-2}\bigl(-(i-1)A_i+\beta_i+A_{i-j+1}\beta_j\bigr)e_i=
-e_1-\sum\limits_{i=3}^nA_ie_i=-e_1'.
\end{align*}
Following a similar computation as for equations \eqref{E:3.1} we derive that $[e_i',x']=-ie_i', \ 1\leq i\leq n$.

Finally, we obtain the  multiplication table of the algebra $R$ given in the in assertion of the theorem.

\textbf{Case 2.} Let $\alpha_1=0$. Then from equalities $\bigl[e_1,[e_1,x]\bigr]=-\bigl[e_1,[x,e_1]\bigr]$ and
$0=\bigl[x,[x,x]\bigr]$ we get $\beta_1=0$ and $\gamma_1=0$, respectively.

Thus, we have the following products:
\[\left\{\begin{aligned}
 {}[e_i,e_1] & =e_{i+1}, && 1\leq i\leq n-1, \\
 [x,e_1] & =\sum\limits_{i=2}^n\alpha_ie_i, && \\
 [e_1,x]& =\sum\limits_{i=2}^n\beta_ie_i, &&\\
 [x,x]& =\sum\limits_{i=2}^n\gamma_ie_i.&&
\end{aligned}\right.\]
In a similar way as for the equations \eqref{E:3.1} it is proved the equality:
$[e_i,x]=\sum\limits_{j=i+1}^n\beta_{j-i+1}e_j$.

Consequently, we have $[e_i,x]\in \langle \{e_{i+1},e_{i+2},\dots,e_n \} \rangle$, i.e. $R^i\subseteq
\langle \{e_i,e_{i+1},\dots,e_n\} \rangle$. Thus, $R^{n+1}=0$ which is a contradiction with the assumption of  non-nilpotency of the algebra $R$. This implies that, in the case of $\alpha_1=0$, there is a not non-nilpotent solvable Leibniz algebra with nilradical $NF_n$.
\end{proof}

Now we are going to clarify the situation when the nilradical is decomposed into a direct sum of two null-filiform ideals of the nilradical.

\begin{thm} \label{thm36} Let $R$ be a solvable Leibniz algebra such that $R=NF_k\oplus NF_s+Q$, where $NF_k\oplus NF_s$ is the nilradical of $R$, $NF_k$ and $NF_s$ are ideals of the  nilradical and $\dim Q=1$. Then $NF_k$ and $NF_s$ are also ideals of the algebra $R$.
\end{thm}

\begin{proof}
Let  $\{e_1,e_2,\dots,e_k\}$  be a basis of $NF_k$, $\{f_1,f_2,\dots,f_s\}$ a basis of $NF_s$ and $\{x\}$ a basis of $Q$. We can assume, without loss of generality, that $k \geq s$, otherwise we can consider $NF_s \oplus NF_k$.

Then due to Theorem~\ref{t21} we have that $\{e_2, e_3, \dots, e_k, f_2, f_3, \dots, f_s\}\subseteq \Ann_r(R)$ and the following products:
\[ [e_i,e_1]=e_{i+1}, \ \ 1\leq i\leq k-1,  \qquad  [f_i,f_1]=f_{i+1}, \ \ 1\leq i\leq s-1\,.\]

Let us introduce the notations:
\[\left\{ \begin{aligned}
{} [x,e_1]& =\sum\limits_{i=1}^k\alpha_ie_i+\sum\limits_{i=1}^s\beta_if_i, \qquad &
 [x,f_1] & =\sum\limits_{i=1}^k\delta_ie_i+\sum\limits_{i=1}^s\gamma_if_i,\\
 [e_1,x] & =\sum\limits_{i=1}^k\lambda_ie_i+\sum\limits_{i=1}^s\sigma_if_i, \qquad &
 [f_1,x] & =\sum\limits_{i=1}^k\tau_ie_i+\sum\limits_{i=1}^s\mu_if_i,\\
[x,x]& =\sum\limits_{i=1}^k\rho_ie_i+\sum\limits_{i=1}^s\xi_if_i.
\end{aligned} \right. \]
From the products
\[0=\bigl[x,[e_1,f_1]\bigr]=\bigl[[x,e_1],f_1\bigr]-\bigl[[x,f_1],e_1\bigr]=\sum\limits_{i=2}^s\beta_{i-1}f_i-\sum\limits_{i=2}^k\delta_{i-1}e_i.\]
we obtain $\beta_i=0, \ 1\leq i\leq s-1$ and $\delta_i=0, \ 1\leq i\leq k-1$.

The equalities $\bigl[e_1,[e_1,x]\bigr]=-\bigl[e_1,[x,e_1]\bigr]$, $\bigl[f_1,[f_1,x]\bigr]=-\bigl[f_1,[x,f_1]\bigr]$ imply that
$\lambda_1=-\alpha_1$, $ \mu_1=-\gamma_1$.

From equalities $0=\bigl[e_1,[x,x]\bigr]=\rho_1e_2, \ 0=\bigl[f_1,[x,x]\bigr]=\xi_1f_2$ we get $\rho_1=\xi_1=0$.

In a similar way as in the proof of Theorem~\ref{thm35}, the following equalities are proved:
\begin{align*}
[e_i,x]&=-i\alpha_1e_i+\sum\limits_{j=i+1}^k\lambda_{j-i+1}e_j, && \ 2\leq i\leq k,\\
[f_i,x]&=-i\gamma_1f_i+\sum\limits_{j=i+1}^s\mu_{j-i+1}f_j, && \ 2\leq i\leq s.
\end{align*}
Summarizing the above restrictions we obtain the following table of multiplication for the algebra $R$:
\begin{equation} \label{E:3.3}
\left\{ \begin{aligned}
{}[e_i,e_1]&=e_{i+1}, && \ 1\leq i\leq k-1, & [f_i,f_1]&=f_{i+1}, && 1\leq i\leq s-1,\\
 [x,e_1]&=\sum\limits_{i=1}^k\alpha_ie_i+\beta_s f_s,&& & [x,f_1]&=\delta_k e_k+\sum\limits_{i=1}^s\gamma_if_i, &&\\
 [e_1,x]&=-\alpha_1e_1+\sum\limits_{i=2}^k\lambda_ie_i+\sum\limits_{i=1}^s\sigma_if_i,&& & [f_1,x]&=\sum\limits_{i=1}^k\tau_ie_i-\gamma_1f_1+\sum\limits_{i=2}^s\mu_if_i, &&\\
 [e_i,x]&=-i\alpha_1e_i+\sum\limits_{j=i+1}^k\lambda_{j-i+1}e_j, && 2\leq i \leq k,& [f_i,x]&=-i\gamma_1e_i+\sum\limits_{j=i+1}^s\mu_{j-i+1}f_j, && 2\leq i\leq s,\\
 [x,x]&=\sum\limits_{i=2}^k\rho_ie_i+\sum\limits_{i=2}^s\xi_if_i. && &  & &&
\end{aligned} \right.
\end{equation}
Below we analyze the different cases  that can appear in terms of the possible values that the parameters  $\alpha_1, \ \gamma_1$ can achieve.

\textbf{Case 1.} Let $\alpha_1=\gamma_1=0$. Then the table of multiplication \eqref{E:3.3} implies
$[e_i,x]\in \langle \{e_{i+1},e_{i+2},\dots,e_k\} \rangle$,
$[f_i,x]\in \langle \{f_{i+1},f_{i+2},\dots,f_s\} \rangle$,
$[e_1,x]\in \langle \{e_2,e_3,\dots,e_k, f_1, f_2, \dots, f_s\} \rangle$ and
$[f_1,x]\in \langle \{e_1, e_2, \dots, e_k,f_2,f_3,\dots,f_s\} \rangle$. The above facts mean that the algebra $R$ is nilpotent, so we get a contradiction with the assumption of non-nilpotency of $R$. Therefore, this case is impossible.

\textbf{Case 2.} Let $\alpha_1\neq0$ and $\gamma_1=0$. Using the following change of basis:

\[e_1'=\frac{1}{\alpha_1}\bigl(\sum\limits_{i=1}^k\alpha_ie_i+\beta_s f_s\bigr), \qquad e_i'=\frac{1}{\alpha_1}\sum\limits_{j=i}^k\alpha_{j-i+1}e_j, \ 2\leq i\leq k, \qquad x'=\frac{1}{\alpha_1}x.\]

We can assume that \[[x,e_1]=e_1.\]

From the  identity \[\bigl[x,[x,e_1]\bigr]=\bigl[[x,x],e_1\bigr]-\bigl[[x,e_1],x\bigr]\]
 we have that
\[e_1=\sum\limits_{i=2}^k\rho_i[e_i,e_1]-[e_1,x]=
\sum\limits_{i=3}^k\rho_{i-1}e_i + e_1 -\sum\limits_{i=2}^k\lambda_ie_i-
\sum\limits_{i=1}^s\sigma_if_i.\] Consequently, $\lambda_2=\sigma_i=0$ for $1\leq i \leq s$ and $\rho_i=\lambda_{i+1}$ for $2\leq i\leq k-1$.

From the identity \[\bigl[f_1,[x,e_1]\bigr]=\bigl[[f_1,x],e_1\bigr]-\bigl[[f_1,e_1],x\bigr]\]
we conclude that $0=\bigl[[f_1,x],e_1\bigr]=\sum\limits_{i=2}^k\tau_{i-1}e_i \Rightarrow \tau_i=0, \ 1\leq i\leq k-1$.

From the identity \[\bigl[x,[x,f_1]\bigr]=\bigl[[x,x],f_1\bigr]-\bigl[[x,f_1],x\bigr],\] we obtain
\begin{align*}
0&=\sum\limits_{i=3}^s\xi_{i-1}f_i-\sum\limits_{i=2}^s\gamma_i[f_i,x]+
\delta_k[e_k,x]=\sum\limits_{i=3}^s\xi_{i-1}f_i-
\sum\limits_{i=2}^s\gamma_i\bigl(\sum\limits_{j=i+1}^s\mu_{j-i+1}f_j\bigr)
-k\delta_ke_k\\& =\sum\limits_{i=3}^s\xi_{i-1}f_i-
\sum\limits_{i=3}^s\bigl(\sum\limits_{j=3}^i\gamma_{j-1}\mu_{i-j+2}\bigr)f_i-k\delta_ke_k
 =\sum\limits_{i=3}^s\bigl(\xi_{i-1}-\sum\limits_{j=3}^i\gamma_{j-1}\mu_{i-j+2}\bigr)f_i-
k\delta_ke_k.
\end{align*}
By comparison of coefficients at the basic elements we deduce that: \[\xi_i=\sum\limits_{j=3}^{i+1}\gamma_{j-1}\mu_{i-j+3}, \ 2\leq
i\leq s-1 \ \mbox{and} \ \delta_k=0.\]
Now we consider the following change of basis:
\[f_1'=f_1+\frac{\tau_k}{k}e_k, \quad \quad f_i'=f_i, \ 2\leq i\leq s.\]
Then we obtain
\[[f_1',x]=[f_1+\frac{\tau_k}{k}e_k,x]=\sum\limits_{i=2}^s\mu_if_i+\tau_k
e_k-\tau_ke_k=\sum\limits_{i=2}^s\mu_if_i=\sum\limits_{i=2}^s\mu_if_i'\]
and
\[[x,f_1']=[x,f_1+\frac{\tau_k}{k}e_k]=[x,f_1]=
\sum\limits_{i=2}^s\gamma_if_i=\sum\limits_{i=2}^s\gamma_if_i'.\]

Thus, we have the following table of multiplication of the algebra $R$
\[\left\{ \begin{aligned}
{} [e_i,e_1]&=e_{i+1}, && 1\leq i\leq k-1, & [f_i,f_1]&=f_{i+1}, && 1\leq i\leq s-1,\\
 [x,e_1]&=e_1,&& & [x,f_1]&=\sum\limits_{i=2}^s\gamma_if_i,&&\\
 [e_1,x]&=-e_1+\sum\limits_{i=2}^k\lambda_ie_i,&& & [f_1,x]&=\sum\limits_{i=2}^s\mu_if_i,&&\\
 [e_i,x]&=-ie_i+\sum\limits_{j=i+2}^k\lambda_{j-i+1}e_j, && 2 \leq i \leq k,& [f_i,x]&=\sum\limits_{j=i+1}^s\mu_{j-i+1}f_j, && 2 \leq i \leq s,\\
 [x,x]&=\sum\limits_{i=2}^k\rho_ie_i+\sum\limits_{i=2}^s\xi_if_i. && &  &   &&
\end{aligned} \right.\]

From the above table of multiplication the following inclusions can be immediately derived:
\[[x,NF_k]\subseteq NF_k, \quad [NF_k,x]\subseteq NF_k, \quad [x,NF_s]\subseteq NF_s, \quad [NF_s,x]\subseteq NF_s.\]
that complete the proof of the assertion established in the theorem for this case.

\textbf{Case 3.} Let $\alpha_1=0$ and $\gamma_1\neq 0$. Due to symmetry of Cases 2 and 3, the proof of the assertion of the theorem
follows similar arguments as in  Case 2.

\textbf{Case 4.} Let $\alpha_1\neq0$ and $\gamma_1\neq 0$. Consider the following change of basis:
\begin{align*}
e_1'& =\frac{1}{\alpha_1}\Bigl(\sum\limits_{i=1}^k\alpha_ie_i+\beta_s f_s\Bigr), & e_i'=\frac{1}{\alpha_1}\sum\limits_{j=i}^k\alpha_{j-i+1}e_j, && \ 2\leq i\leq k, & \\
f_1'& =\frac{1}{\gamma_1}\Bigl(\sum\limits_{i=1}^s\gamma_if_i+\delta_k e_k\Bigr), & f_i'=\frac{1}{\gamma_1}\sum\limits_{j=i}^k\gamma_{j-i+1}f_j, && \ 2\leq i\leq s, & \quad x'=\frac{1}{\alpha_1}x.
\end{align*}
Then we derive
\begin{align*}
[x',e_1']&=\bigl[\frac{1}{\alpha_1}x,\frac{1}{\alpha_1}\bigl(\sum\limits_{i=1}^k\alpha_ie_i+\beta_s
f_s\bigr)\bigr]=\frac{1}{\alpha_1^2}\alpha_1[x,e_1]=\frac{1}{\alpha_1}[x,e_1]=e_1', \\
[x',f_1']&=\bigl[\frac{1}{\alpha_1}x,\frac{1}{\gamma_1}\bigl(\sum\limits_{i=1}^s\gamma_if_i+\delta_k e_k\bigr)\bigr]=\frac{1}{\alpha_1\gamma_1}\gamma_1[x,f_1]=\frac{\gamma_1}{\alpha_1}f_1'.
\end{align*}
From the identity $\bigl[x,[x,e_1]\bigr] = \bigl[[x,x],e_1\bigr] - \bigl[[x,e_1],x\bigr]$ we deduce:
\[e_1=\sum\limits_{i=2}^k\rho_i[e_i,e_1]-[e_1,x]=
\sum\limits_{i=3}^k\rho_{i-1}e_i+ \alpha_1 e_1 -\sum\limits_{i=2}^k\lambda_ie_i-
\sum\limits_{i=1}^s\sigma_if_i.\]
Therefore, $\alpha_1 = 1, \lambda_1=-1, \ \lambda_2=\sigma_i=0,  \ 1\leq i \leq s$ and $\rho_i=\lambda_{i+1}, \ 2\leq i\leq k-1$.

Expanding the identity $\bigl[x,[x,f_1]\bigr]=\bigl[[x,x],f_1\bigr]-\bigl[[x,f_1],x\bigr]$, we derive the equalities:
\[\Bigl(\frac{\gamma_1}{\alpha_1}\Bigr)^2f_1=
\sum\limits_{i=2}^s\xi_i[f_i,f_1]-\frac{\gamma_1}{\alpha_1}[f_1,x]=
\sum\limits_{i=3}^s\xi_{i-1}f_i-\frac{\gamma_1}{\alpha_1}\sum\limits_{i=1}^s\mu_if_i-
\frac{\gamma_1}{\alpha_1}\sum\limits_{i=1}^k\tau_ie_i \]
from which we have $\mu_1=-\frac{\gamma_1}{\alpha_1}, \ \mu_2=\tau_i=0, \ 1 \leq
i \leq k$ and $\xi_i= \frac{\gamma_1}{\alpha_1} \mu_{i+1}, \ 2\leq i\leq s-1$.

Finally, we obtain the following products of basis elements in the algebra $R$:
\[\left\{\begin{aligned}
{} [e_i,e_1]& =e_{i+1}, && 1\leq i\leq k-1, & [f_i,f_1]& =f_{i+1}, && 1\leq i\leq s-1,\\
 [x,e_1]&=e_1,&& & [x,f_1]&=\frac{\gamma_1}{\alpha_1}f_1,&&\\
[e_1,x]&=-e_1+\sum\limits_{i=3}^k\lambda_ie_i,&& & [f_1,x]&=-\frac{\gamma_1}{\alpha_1}f_1+\sum\limits_{i=3}^s\mu_if_i,&&\\
 [x,x]&=\sum\limits_{i=2}^k\rho_ie_i+\sum\limits_{i=2}^s\xi_if_i. && &  &  &&
\end{aligned} \right.\]
These products are sufficient in order to check the inclusions:
\[[x,NF_k]\subseteq NF_k, \  [NF_k,x]\subseteq NF_k, \ [x,NF_s]\subseteq NF_s, \ [NF_s,x]\subseteq NF_s.\]
Thus, the ideals $NF_k$ and $\ NF_s$ of the nilradical are also ideals of the algebra.
\end{proof}

Now we are going to study  solvable Leibniz algebras with nilradical $NF_k\oplus NF_s$ and with one-dimensional complementary vector  space. Due to Theorem~\ref{thm36} we can assume that $NF_k$ and $NF_s$ are ideals of the algebra.

\begin{thm}  \label{thm37} Let $R$ be a solvable Leibniz algebra such that $R=NF_k\oplus NF_s+Q$, where $NF_k\oplus NF_s$ is the nilradical of $R$ and $\dim Q=1$. Let us assume that $\{ e_1, e_2, \dots, e_k \}$ is a basis of $NF_k$, $\{ f_1, f_2, \dots, f_s \}$ is a basis of $NF_s$ and $\{ x \}$ is a basis of $Q$. Then the algebra $R$ is isomorphic to one of the following pairwise non-isomorphic algebras:
\[R(\alpha):\left\{ \begin{aligned}
{}[e_i,e_1]&=e_{i+1}, && 1\leq i\leq k-1, & [f_i,f_1]&=f_{i+1}, && 1\leq i\leq s-1,\\
[x,e_1]&=e_1,&& & [x,f_1]&=\alpha f_1, && \alpha\neq0 \\
[e_i,x]&=-ie_i, && 1\leq i \leq k,& [f_i,x]&=-i\alpha f_i, && 1 \leq i \leq s.
\end{aligned} \right.\]
\[R(\beta_2, \beta_3, \dots , \beta_s, \gamma):\left\{ \begin{aligned}
{}[e_i,e_1]&=e_{i+1}, && 1\leq i\leq k-1, & [f_i,f_1]&=f_{i+1}, && 1\leq i\leq s-1,\\
[x,e_1]&=e_1,&&  &[f_i,x]&=\sum\limits_{j=i+1}^s\beta_{j-i+1}f_j, && 1 \leq i \leq s,\\
[e_i,x]&=-ie_i, && 1 \leq i \leq k,& [x,x]&=\gamma f_s.&&
\end{aligned} \right.\]
in the second family of algebras the first non-zero element of the set $(\beta_2, \beta_3, \dots , \beta_s, \gamma)$ can be assumed equal to 1.
\end{thm}
\begin{proof}  Firstly, we note that the algebras $NF_k+Q$ and $NF_s+Q$ are not simultaneously nilpotent. Indeed, if they  are both nilpotent, then  we have:
\[\begin{aligned}
{} [e_i,e_1] & \in  \langle \{e_{i+1},\dots, e_k \} \rangle, && 1\leq i\leq k-1, &[f_i,f_1] &\in \langle \{f_{i+1}, \dots, f_s \} \rangle, &&  1\leq i\leq s-1,\\
[x,e_1] & \in \langle \{ e_2, e_3, \dots, e_k \} \rangle, &&  &{}[x,f_1] &\in \langle \{f_2, f_3, \dots, f_s \} \rangle, &&\\
 [e_i,x]& \in \langle \{ e_{i+1}, \dots, e_k \} \rangle, &&  1 \leq i \leq k-1, &[f_j,x]& \in \langle \{f_{j+1}, \dots, f_s \} \rangle, && 2 \leq i \leq s-1,
\end{aligned}\]

From the equalities $0=\bigl[e_1,[x,x]\bigr], \ 0=\bigl[f_1,[x,x]\bigr]$ we conclude that:
\[\, [x,x]\in \langle \{e_2, e_3, \dots, e_k, f_2, f_3, \dots, f_s \} \rangle.\]
Therefore, $R^2\subseteq \{e_2, e_3, \dots, e_k, f_2, f_3, \dots, f_s\}$. Moreover, we have $R^i\subseteq \{e_i, e_{i+1}, \dots, e_k, f_i, f_{i+1}, \dots, f_s\}$, which implies that $R^{\max\{{k,s}\}+1}=\{0\}$, i.e. we get a contradiction with the assumption of non-nilpotency of the algebra $R$.
Hence, the algebras $NF_k+Q$ and $NF_s+Q$ cannot be simultaneously nilpotent.

Without loss of generality, we can assume that algebra $NF_k+Q$ is non-nilpotent.

We take the quotient algebra by ideal $NF_s$, then $R/NF_s\cong
\overline{NF_k}+\overline{Q}$. Thanks to Theorem~\ref{thm35} the structure of algebra
$\overline{NF_k}+\overline{Q}$ is known. Namely,
\begin{equation} \label{E:3.4}
\left\{\begin{aligned}
{} [\overline{e_i},\overline{e_1}]& =\overline{e_{i+1}}, && 1\leq i\leq k-1, \\
 [\overline{x},\overline{e_1}]&=\overline{e_1}, && \\
  [\overline{e_i},\overline{x}]&=-i\overline{e_i}, && 1\leq i\leq k.
\end{aligned} \right.
\end{equation}
Using the fact that $NF_k$ and $\ NF_s$ are ideals of $R$ and having in mind the table of multiplication \eqref{E:3.4}, we have that:
\begin{equation} \label{E:3.5}
\left\{ \begin{aligned}
{} [e_i,e_1]&=e_{i+1}, && \ 1\leq i\leq k-1,  \qquad &[f_i,f_1]&=f_{i+1}, && \ 1\leq i\leq s-1,\\
[x,e_1]& =e_1,&&   \qquad &[x,f_1]& =\sum\limits_{i=1}^s\alpha_if_i,&&\\
 [e_i,x]&=-ie_i, && \ 1\leq i\leq k, \qquad  &[f_1,x]&=\sum\limits_{i=1}^s\beta_if_i,&&\\
 &  &&  &[x,x]&=\sum\limits_{i=1}^s\gamma_if_i.   &&
\end{aligned} \right.
\end{equation}
If $\alpha_1\neq 0$, then in a similar way  as  the Case 1 of Theorem~\ref{thm35} we obtain the family of algebras $R(\alpha)$, where $\alpha\neq0$.

The non-isomorphy of two algebras in the family $R(\alpha)$ with different values of parameter $\alpha$ easily can be determined by a general change of basis and considering the expansion of the product
$[x',f_1']$ in both bases.

Now consider $\alpha_1=0$. Then by the change of basis
\[x'=x - (\alpha_2f_1 + \alpha_3f_2 + \dots +
\alpha_sf_{s-1}),\] we can suppose  $[x, f_1]=0$.

From the identity  $\bigl[f_1,[f_1,x]\bigr] = \bigl[[f_1,f_1],x\bigr]-\bigl[[f_1,x],f_1\bigr]$ we get $\beta_1=0$.

Similarly to the proof of equations \eqref{E:3.1}, we can  prove that
$[f_i,x]=\sum\limits_{m=i+1}^s\beta_{m-i+1}f_j, \ 1\leq i\leq s$.

The identity $\bigl[x,[f_1,x]\bigr]=\bigl[[x,f_1],x\bigr]-\bigl[[x,x],f_1\bigr]$ implies the following chain of equalities:
\[0=-\bigl[[x,x],f_1\bigr]=-\sum\limits_{m=3}^s\gamma_{m-1}f_m \, .\]
Consequently, $\gamma_i=0, \ 2\leq i\leq s-1$.

Thus, we obtain the products of the family $R(\beta_2, \beta_3, \dots , \beta_s, \gamma)$
\[\left\{\begin{aligned}
{}[f_i,f_1]&=f_{i+1},  && 1\leq i\leq s-1,\\
[f_i,x]& =\sum\limits_{m=i+1}^s\beta_{m-i+1}f_m, && 1\leq i\leq s,\\
\,[x,x]&=\gamma_s f_s. &&
\end{aligned} \right.\]
Now we are going to study the isomorphism inside the family $R(\beta_2, \beta_3, \dots , \beta_s, \gamma)$.

Taking into account that, under general basis transformation, the products \eqref{E:3.5} should not be changed, we conclude that it is sufficient to take the following change of basis:
\[f_i'=A_1^{i-1}\sum\limits_{j=i}^sA_{j-i+1}f_j, \ (A_1\neq0) \ 1\leq i\leq s, \quad x'=x.\]
Then we have
\[[f_1',x']= \displaystyle \sum_{i=1}^s A_i[f_i,x]= \displaystyle \sum_{i=1}^{s-1} A_i \Bigl( \displaystyle \sum_{j=i+1}^s \beta_{j-i+1} f_j \Bigr)= \displaystyle \sum_{i=2}^s \Bigl( \sum_{j=1}^{i-1} A_j B_{i-j+1} \Bigr) f_i\]
On the other hand
\[[f_1',x']= \displaystyle \sum_{i=2}^s \beta_i'f_i'= \displaystyle \sum_{i=1}^{s-1} A_1^i \beta'_{i+1} \Bigl( \displaystyle \sum_{j=1}^{s-i} A_j f_{i+j} \Bigr) = \displaystyle \sum_{i=2}^s \Bigl( \displaystyle \sum_{j=1}^{i-1} A_1^j A_{i-j} \beta'_{j+1} \Bigr) f_i\]
Comparing coefficients at the basic elements we deduce that:
\[\sum_{i=1}^{k-1} A_i \beta_{k-i+1} = \sum_{i=1}^{k-1} A_1^i A_{k-i} \beta'_{i+1}, \quad k = 2, 3, \dots, s\]
From these systems of equations easily it follows:
\[\beta_i'=\frac{\beta_i}{A_1^{i-1}}, \quad  \ 2\leq i\leq s.\]
If we consider
\[\gamma_s' A_1^sf_s=\gamma_s' f_s'=[x',x']=[x,x]=\gamma_s f_s\]
then we obtain that:
\[\ \gamma_s'=\frac{\gamma_s}{A_1^s}.\]
It is easy to see that by choosing an adequate value for the parameter $A_1$, then the first non-zero element of the set $(\beta_2, \beta_3, \dots, \beta_s, \gamma)$ can be assumed equal to 1.

Therefore, two algebras $R(\beta_2, \beta_3, \dots , \beta_s, \gamma)$ and $R(\beta_2', \beta_3', \dots , \beta_s', \gamma')$ with different set of parameters are not isomorphic.

For a given parameters $\alpha$ and $\beta_2, \beta_3, \dots , \beta_s, \gamma$, the algebras $R(\alpha)$ and $R(\beta_2, \beta_3, \dots , \beta_s, \gamma)$ are not isomorphic because \[k+s= \dim R(\alpha)^2\neq \dim R(\beta_2, \beta_3, \dots , \beta_s, \gamma)^2=k+s-1.\]
\end{proof}
\begin{rem} In the case when all coefficients $(\beta_2, \beta_3, \dots , \beta_s, \gamma)$ are equal to zero we have the split algebra $(NF_k+Q) \oplus NF_s$. Therefore, in non-split case, we can always assume that $(\beta_2, \beta_3, \dots, \beta_s, \gamma) \neq (0, 0, 0, \dots, 0)$.
\end{rem}

 Now, by  an induction process, we are going to  generalize Theorem~\ref{thm37} to the case when the  nilradical is a direct sum (greater than 2) of several copies of null-filiform ideals.

\begin{thm} Let $R$ be a solvable Leibniz algebra such that $R=NF_{n_1}\oplus NF_{n_2}\oplus \dots \oplus NF_{n_s}+Q$, where $NF_{n_1}\oplus NF_{n_2}\oplus \dots\oplus NF_{n_s}$ is the nilradical of $R$ and $\dim  Q=1$. Let us assume that $\{e_1^i, e_2^i, \dots, e_{n_i}^i \}$ is a basis of $NF_{n_i}, 1 \leq i \leq j'$, and $\{f_1^k, f_2^k, \dots, f_{n_k}^k \}$ is a basis of $NF_{n_{j'+k}}, 1 \leq k \leq k'$, and $\{ x \}$ is a basis of $Q$. Then the algebra $R$ is isomorphic to one of the following pairwise non-isomorphic algebras:
\begin{equation} \label{E:3.6}
R_{j',k'}:\left\{\begin{aligned}
{} [e_i^j,e_1^j]&=e_{i+1}^j, &&  1\leq i\leq n_j-1, & [f_i^k,f_1^k]&=f_{i+1}^k, && 1\leq i\leq n_k-1,\\
 [x,e_1^j]&=\delta^je_1^j, && \delta^j\neq0
& [f_i^k,x]&=\sum\limits_{m=i+1}^{n_k}\beta_{m-i+1}^kf_m^k, && 1\leq i\leq n_k,\\
 [e_i^j,x]&=-i\delta^je_i^j, && 1\leq i\leq n_j, & [x,x]&=\sum\limits_{m=1}^k\gamma^mf_{n_m}, &&
\end{aligned} \right.
\end{equation}
where $j'+k'=s$, $1 \leq j \leq j', 1 \leq k \leq k'$, $j' \neq 0$ and
$\delta^1 =1$.
Moreover, the first non-zero elements among set $(\beta_2^k, \beta_3^k, \dots , \beta_{n_k}^k, \gamma^k)$  can be assumed  equal to 1.
\end{thm}
\begin{proof} By induction on $s$:

If $s=1$, then $j'=1, \ k'=0$, so $R_{1, 0}$ is the  algebra given in  Theorem~\ref{thm35}.

If $s=2$, then we have two cases: either $j'=2, \ k'=0$ or $j'=1, \ k'=1$, which were considered in Theorem~\ref{thm36}. Namely, we have two families of algebras:  $R(\alpha)$, which corresponds to $R_{2, 0}$, and $R(\beta_2, \beta_3, \dots, \beta_s, \gamma)$, which corresponds to $R_{1,1}$.

Let us assume that the theorem is true for $s$ and we shall prove it for $s+1$.

Let $R=NF_{n_1}\oplus NF_{n_2}\oplus \dots \oplus NF_{n_s}\oplus
NF_{n_{s+1}}+Q$. We consider the quotient algebra by $NF_{n_{s+1}}$, i.e. $R/NF_{n_{s+1}}\cong \overline{NF_{n_1}}\oplus
\overline{NF_{n_2}}\oplus \dots \oplus\overline{
N_{nF_s}}+\overline{Q}$. Then we get the table of multiplication given in \eqref{E:3.6}.

Note that the table of multiplication for the algebra $R$ can be obtained from \eqref{E:3.6} by adding the products
\begin{align*}
 [e_i^{s+1},e_1^{s+1}]& = e_{i+1}^{s+1}, && 1\leq i\leq n_{s+1}-1, \\
 [x,e_1^{s+1}] & = \sum\limits_{m=1}^{n_{s+1}}\alpha_m^{s+1}e_m^{s+1}, && \\
 [e_1^{s+1},x] & = \sum\limits_{m=1}^{n_{s+1}}\beta_m^{s+1}e_m^{s+1}, &&\\
 [x,x] & =\sum\limits_{m=1}^{n_{s+1}}\gamma_m^{s+1}e_m^{s+1}. &&
\end{align*}
If $\alpha_1^{s+1}\neq0$, then in an analogous way as  in proof of Theorem~\ref{thm35} we derive
\begin{align*}
 [e_i^{s+1},e_1^{s+1}] & =e_{i+1}^{s+1}, && 1\leq i\leq n_{s+1}-1, \\
 [x,e_1^{s+1}]& =\alpha_{s+1}^{s+1}e_1^{s+1}, && \\
 [e_i^{s+1},x] & =-i\alpha^{s+1}e_i^{s+1}, && 1\leq i\leq n_{s+1}.
\end{align*}
Therefore we get the algebra $R_{j'+1,k'}$.

If $\alpha_1^{s+1}=0$, then by similar arguments  as in Theorem~\ref{thm37} we obtain
\begin{align*}
 [e_i^{s+1},e_1^{s+1}] & =e_{i+1}^{s+1}, && 1\leq i\leq n_{s+1}-1, \\
 [e_i^{s+1},x] & =\sum\limits_{m=i+1}^{n_{s+1}}\beta_{m-i+1}^{s+1}f_m^{s+1}, &&  1\leq i\leq n_{s+1},\\
 [x,x] & = \sum\limits_{m=1}^k\gamma^mf_{n_m} + \gamma^{s+1}f_{n_{s+1}}^{s+1}.  &&
\end{align*}
Setting $f_{i-1}^{k'+1} = e_{i-1}^{s+1}$ we get the family of algebras
$R_{j',k'+1}$.

Non-isomorphy of two algebras of the family $R_{j',k'}$ with different values of parameters is carrying out in a similar way as in the proof of Theorem~\ref{thm37}.
\end{proof}

In fact, due to Theorem~\ref{t31} the complementary vector space in case when the  nilradical of a solvable Leibniz algebra is a direct sum of $s$ copies of null-filiform ideals has dimension not grater than $s$. By taking direct sum of ideals $NF_i+Q_i$ and $NF_k\oplus \dots \oplus NF_s$, where $1 \leq i \leq k-1, \ k\leq s$, we can construct a solvable Leibniz algebra whose nilradical is $NF_1\oplus\dots\oplus NF_s$ and whose complementary vector space is $k$-dimensional for each $k$ ($k\leq s$).

\section*{Acknowledgements}

First and second  authors were supported by Ministerio de
Ciencia e Innovaci\'on, Grant MTM2009-14464-C02 (European FEDER
support included) and by Xunta de Galicia, Grant Incite09 207 215 PR.


\begin{thebibliography}{99}

\bibitem{AlAyBa}
S.~A. Albeverio, S.~A. Ayupov, B.~A. Omirov, Cartan subalgebras, weight spaces,
  and criterion of solvability of finite dimensional {L}eibniz algebras, Rev.
  Mat. Complut. 19~(1) (2006) 183--195.


\bibitem{AnCaGa1}
J.~M. Ancochea~Berm{\'u}dez, R.~Campoamor-Stursberg,
  L.~Garc{\'{\i}}a~Vergnolle, Indecomposable {L}ie algebras with nontrivial
  {L}evi decomposition cannot have filiform radical, Int. Math. Forum 1~(5-8)
  (2006) 309--316.

\bibitem{AnCaGa2}
J.~M. Ancochea~Berm{\'u}dez, R.~Campoamor-Stursberg,
  L.~Garc{\'{\i}}a~Vergnolle,
 {Classification of
  {L}ie algebras with naturally graded quasi-filiform nilradicals}, J. Geom.
  Phys. 61~(11) (2011) 2168--2186.
  

\bibitem{AyOm1}
S.~A. Ayupov, B.~A. Omirov, On {L}eibniz algebras, in: Algebra and operator
  theory ({T}ashkent, 1997), Kluwer Acad. Publ., Dordrecht, 1998, pp. 1--12.

\bibitem{AyOm2}
S.~A. Ayupov, B.~A. Omirov, {On
  some classes of nilpotent {L}eibniz algebras}, Siberian Math. J. 42~(1)
  (2001) 15--24.
  
\bibitem{Bar}
D.~W. Barnes, On {L}evi's theorem for {L}eibniz algebras, arxiv:1109.1060v1.


\bibitem{BoPaPo}
V.~Boyko, J.~Patera, R.~Popovych,
  {Invariants of solvable
  {L}ie algebras with triangular nilradicals and diagonal nilindependent
  elements}, Linear Algebra Appl. 428~(4) (2008) 834--854.

\bibitem{Cam}
R.~Campoamor-Stursberg,
  {Solvable {L}ie
  algebras with an {$ \mathbb{N}$}-graded nilradical of maximal nilpotency degree
  and their invariants}, J. Phys. A 43~(14) (2010) 145202, 18.
  
\bibitem{Jac}
N.~Jacobson, Lie algebras, Interscience Tracts in Pure and Applied Mathematics,
  No. 10, Interscience Publishers (a division of John Wiley \& Sons), New
  York-London, 1962.

  
\bibitem{Lod}
J.-L. Loday, Une version non commutative des alg\`ebres de {L}ie: les
  alg\`ebres de {L}eibniz, Enseign. Math. (2) 39~(3-4) (1993) 269--293.
  
\bibitem{Mal}
A.~I. Malcev, Solvable {L}ie algebras, Amer. Math. Soc. Translation 1950~(27)
  (1950).
  
\bibitem{Mub}
G.~M. Mubarakzjanov, On solvable {L}ie algebras, Izv. Vys\v s. U\v cehn. Zaved.
  Matematika 1963~(no 1 (32)) (1963) 114--123.


\bibitem{NdWi}
J.~C. Ndogmo, P.~Winternitz,
  {Solvable {L}ie algebras with
  abelian nilradicals}, J. Phys. A 27~(2) (1994) 405--423.
  
\bibitem{Omi}
B.~A. Omirov, {Conjugacy
  of {C}artan subalgebras of complex finite-dimensional {L}eibniz algebras}, J.
  Algebra 302~(2) (2006) 887--896.

\bibitem{SnKa}
L.~{\v{S}}nobl, D.~Kar{\'a}sek,
 {Classification of solvable
  {L}ie algebras with a given nilradical by means of solvable extensions of its
  subalgebras}, Linear Algebra Appl. 432~(7) (2010) 1836--1850.

\bibitem{SnWi}
L.~{\v{S}}nobl, P.~Winternitz,
  {A class of solvable
  {L}ie algebras and their {C}asimir invariants}, J. Phys. A 38~(12) (2005)
  2687--2700.

\bibitem{TrWi}
S.~Tremblay, P.~Winternitz,
  {Solvable {L}ie algebras
  with triangular nilradicals}, J. Phys. A 31~(2) (1998) 789--806.
  
  \bibitem{Ver}
M.~Vergne, Cohomologie des alg\`ebres de {L}ie nilpotentes. {A}pplication \`a
  l'\'etude de la vari\'et\'e des alg\`ebres de {L}ie nilpotentes, Bull. Soc.
  Math. France 98 (1970) 81--116.

\bibitem{WaLiDe}
Y.~Wang, J.~Lin, S.~Deng,
  {Solvable {L}ie algebras
  with quasifiliform nilradicals}, Comm. Algebra 36~(11) (2008) 4052--4067.
\end{thebibliography}

\end{document}